\documentclass[]{interact}

\usepackage[caption=false]{subfig}

\usepackage[numbers,sort&compress]{natbib}
\bibpunct[, ]{[}{]}{,}{n}{,}{,}

\usepackage{subfig}
\usepackage{graphicx}
\usepackage{xcolor}
\usepackage{soul}
\usepackage{lscape}

\theoremstyle{plain}
\newtheorem{theorem}{Theorem}[section]
\newtheorem{lemma}[theorem]{Lemma}

\theoremstyle{definition}

\newtheorem{assumption}[theorem]{Assumption}
\newtheorem{algorithm}[theorem]{Algorithm}
\usepackage{algorithm}
\usepackage{algpseudocode}
\usepackage{amsmath}
\usepackage{amsthm}
\usepackage{amsfonts}
\theoremstyle{remark}

\usepackage{booktabs}
\usepackage{multirow}

\begin{document}


\title{
\begin{center} \Large
A family of spectral conjugate gradient algorithms derived by least-squares approximations based on a modified quasi--Newton update with application to a revised robust binary classification model
\end{center}}

\author{
\begin{center}
\name{Saman Babaie--Kafaki\textsuperscript{a}{*}\thanks{\textsuperscript{*}Corresponding author (Email: \texttt{saman.babaiekafaki@unibz.it})}, 
Maryam Khoshsimaye--Bargard\textsuperscript{b}, 
and Ahmad Mousavi\textsuperscript{c}}
\affil{\textsuperscript{a}\emph{Faculty of Engineering, Free University of Bozen--Bolzano, 39100 Bolzano, Italy}}
\affil{\textsuperscript{b}\emph{Department of Mathematics, Faculty of Mathematics, Statistics and Computer Science, Semnan University, Semnan, Iran}}
\affil{\textsuperscript{c}\emph{Department of Mathematics and Statistics, American University, Washington, D.C., USA}}
\end{center}
}

\maketitle

\begin{abstract}
We develop a spectral three-term modification of the classic Hestenes--Stiefel conjugate gradient algorithm, preserving its anti-jamming characteristic and, simultaneously, taking care of the sufficient descent property. We discuss how a modified secant equation can be extracted from our modification scheme, yielding a memoryless BFGS  updating formula. Then, the spectral parameter of our method is obtained by steering its direction toward the given BFGS direction within a least-squares context. Using our technical improvements, we outline the general framework of our algorithm and discuss its theoretical features, including the descent and convergence properties, without the convexity assumption. We put our algorithm to the test in comparison with the three other conjugate gradient algorithms on a set of CUTEr unconstrained optimization test models, comparing the outputs using the Dolan--Mor{\'e} measure. Next, we provide a concise evaluation of the results, highlighting the practical advantages of our algorithm. As a real-world case study, we introduce a reduced quadratic surface SVM with the rescaled loss for robust binary classification and apply the proposed algorithm to assess its accuracy and training time against several other SVMs.\\

\noindent \textbf{Keywords:} Nonlinear programming, conjugate gradient method, quasi{--}Newton update, least-squares model, quadratic surface SVM, robust binary classification.\\

\noindent \textbf{MSC 2020:} 90C06; 49M37; 68T05.

\end{abstract}

\section{Introduction}

Due to abundant applications of nonlinear optimization, especially in the current practical trends of machine learning and data mining, a huge range of studies has been recently dedicated to developing efficient algorithms for nonlinear decision-making models. The experimental observations in the vast majority of the real-world cases show that the need for memoryless (limited-memory) schemes to handle the current high-dimensional optimization models is exacerbated. Being now a viable and prevalent plan in practical applications, memoryless approaches should be designed by a conscious and judicious trade-off between accuracy and efficiency, two factors that often place in opposition to each other. It is also a reasonable matter of routine to keep looking at the diversity and inclusion in the optimization tools to flexibly address various models arising continuously and frequently in different data analysis disciplines.

The huge growth in the dimension of practical nonlinear programming models by the beginning of the current century pushed scholars to broaden their horizons and change their classic way of thinking toward memoryless Hessian approximations. A careful readout of the literature reveals a great deal of interest in devising memoryless algorithms to address the well-known unconstrained optimization model, i.e.
\begin{equation}\label{UP}
  \min_{x\in \mathbb{R}^n}\ f(x),
\end{equation}
as the core of nonlinear optimization problems, where $f$ is called the cost function, here assumed to be smooth. In such studies, mainly a meaningful plan can be observed to impede the expected computational consequences of directly dealing with the second-order information of the model, which is often dense and complicated.

Among the strategies that implicitly benefit from the Hessian information for solving (\ref{UP}), there exist the conjugate gradient (CG) algorithms with the iterations
\begin{equation}\label{xk}
x_{k+1}=x_{k}+s_{k},\qquad \forall k\geq0,
\end{equation}
starting from some point $x_{0}\in\mathbb{R}^n$, where the step $s_k$ is essentially determined as a multiplication of the two following elements \cite{Andrei}:
\begin{itemize}
\item [($i$)] the search direction $d_k$ defined by
\begin{eqnarray}\label{dk}
d_{0}=-g_{0},\qquad d_{k+1}=-g_{k+1}+\beta_{k}d_k,\qquad \forall k\geq0,
\end{eqnarray}
where $g_k=\nabla f(x_k)$ and $\beta_{k}$ is referred to as the CG parameter;

\item [($ii$)] the step length $\alpha_k$ obtained through a line search in the direction $d_k$.
\end{itemize}

One of the measures that has significantly influenced the development of distinct CG techniques during the past decades is the way of setting the CG parameter $\beta_{k}$ \cite{HagerZhang1}. Among the classic practically effective choices of $\beta_{k}$, we can mention that introduced by Polak, Ribi\`{e}re and Polyak (PRP)  \cite{{PolakRibiere},{Polyak}}, i.e. $\beta_{k}^\textrm{PRP}={\left(g_{k+1}^Ty_{k}\right)}/{\left\Vert g_{k}\right\Vert^2}$, as well as that of Hestenes and Stiefel (HS) \cite{HestenesStiefel}, i.e. $\beta_{k}^\textrm{HS}={\left(g_{k+1}^Ty_{k}\right)}/{\left(d_{k}^Ty_{k}\right)}$, where $y_{k} = g_{k+1} - g_{k}$ denotes the gradient displacement vector, and $\|.\|$ represents the Euclidean norm.
However, even under convexity suppositions, the corresponding CG methods fail to guarantee the sufficient descent condition, defined by
\begin{equation}\label{sufficient descent condition 1}
g_{k}^Td_{k}\leq-\varsigma\Vert g_{k}\Vert^2,\qquad \forall k\geq0,
\end{equation}
where $\varsigma$ is a positive constant \cite{HagerZhang1}. This essential defect pushed the researchers to create some clear theoretical foundations to achieve (\ref{sufficient descent condition 1}) by devising extended versions of the PRP and HS methods. Among such insightful plans, the Dai--Liao (DL) approach \cite{Babaie} and the three{-}term expansions \cite{SugikiNarushimaYabe} thoroughly revolutionized the CG literature in computational and theoretical respects. In particular, the basic version of the DL parameter is defined by $\beta_{k}^\textrm{DL}=\beta_{k}^\textrm{HS}-t{\left(g_{k+1}^Ts_k\right)}/{\left(d_k^Ty_k\right)}$ \cite{DaiLiao}, with $t$ as a nonnegative parameter, which could be appropriately set to ensure (\ref{sufficient descent condition 1})  \cite{Babaie}. Moreover, Zhang, Zhou and Li (ZZL) \cite{ZhangZhouLi} developed the following three-term version of the HS search direction:
\begin{equation}\label{dZZLHS}
d_{0}=-g_{0},\qquad d_{k+1}^\textrm{ZZL}=-g_{k+1}+\beta_{k}^\textrm{HS}d_{k}-\dfrac{g_{k+1}^Td_{k}}{d_{k}^Ty_{k}}y_{k},\qquad \forall k\geq0,
\end{equation}
which satisfies an equality version of \eqref{sufficient descent condition 1} with $\varsigma = 1$.

Here, inspired by the study by Chen, Kuang, and Yan \cite{ChenKuangYan} on the PRP method, we first develop a modified version of the HS parameter by expanding the denominator of $\beta_{k}^\textrm{HS}$. Then, to satisfy the sufficient descent condition \eqref{sufficient descent condition 1}, a three-term extension of the given (CG) direction is proposed. To more directly benefit from the second-order information of the cost function, a spectral parameter is embedded on the direction of the algorithm \cite{BirginMartinez}. Afterward, by introducing a modified secant equation, the spectral parameter is derived by pushing the given direction toward that of the well-known (memoryless) quasi--Newton (QN) updating formulas in a least-squares framework. As another major part of our study, we present a concise assessment of the results, emphasizing the practical benefits of our algorithm. That is, as a real-world case study, we propose a reduced quadratic surface Support Vector Machine (SVM) equipped with the rescaled loss for robust binary classification \cite{Aminifard2027}. Then, we employ the proposed algorithm to evaluate its classification accuracy and training time in comparison with several other SVM models.

The structure of this work is as follows: Section \ref{main} is devoted to the development of a spectral method within the algorithmic context of three-term CG strategies. Section \ref{algorithm} contains the necessary theoretical analysis, such as establishing the sufficient descent condition and the global convergence of the proposed algorithm. Section \ref{numerical} outlines the effects observed from comparative computational experiments carried out on a collection of standard test functions. Our revised quadratic surface SVM model for robust binary classification is presented and subsequently evaluated in Section \ref{sec:rqsvm}. Ultimately, Section \ref{conclusions} encapsulates the general conclusions of this study.

\section{Developing a class of spectral three-term conjugate gradient techniques}\label{main}

It is widely discussed in the literature that, despite the capability to automatically restart when the iterations are stuck far from a solution of (\ref{UP}) in the sense of jamming \cite{BabaieMirhoseiniAminifard}, leading to favorable computational outputs, the PRP and HS methods suffer from some poor theoretical properties \cite{HagerZhang1}. As already mentioned, the theoretical weaknesses of the PRP and HS techniques are mainly rooted in the possibility of generating uphill (nondescent) search directions. Amidst the efforts to combat this essential defect \cite{Babaie}, recently Chen, Kuang and Yan \cite{ChenKuangYan} suggested the following modified PRP parameter:
\begin{equation*}
\beta_k^\textrm{MPRP}=\dfrac{g_{k+1}^Ty_{k}}{\Vert g_{k}\Vert^2+\mu\Vert y_{k}\Vert\Vert d_{k}\Vert},
\end{equation*}
in which $\mu>1$ is a constant. Actually, in contrast to the similar studies conducted in \cite{{WeiYaoLiu},{YaoWeiHuang}} with some modifications in the numerator of $\beta_k^\textrm{PRP}$ or $\beta_k^\textrm{HS}$, Chen, Kuang and Yan \cite{ChenKuangYan} only extended the denominator with the purpose of guaranteeing the sufficient descent condition \eqref{sufficient descent condition 1}. By such a plan, they made it possible to maintain the anti-jamming (restart) characteristic for the MPRP method as well. It should be highlighted that MPRP has been shown to be globally convergent for nonconvex cost functions when using either Wolfe or the Armijo line search strategies \cite{Andrei}.

Now, recalling the structural similarity of the PRP and HS parameters, and inspired by the argument presented in \cite{ChenKuangYan}, we propose a modified version of the HS parameter as
\begin{equation}\label{dMHS}
\beta_k^\textrm{MHS}=\dfrac{g_{k+1}^Ty_{k}}{d_{k}^Ty_{k}+\zeta\Vert y_{k}\Vert\Vert d_{k}\Vert},
\end{equation}
in which $\zeta>0$ is called the MHS constant. Meanwhile, to make it possible to achieve the sufficient descent condition, we develop the following modified spectral three-term version of the MHS direction:
\begin{equation}\label{dSTTMHS}
d_{k+1}^\textrm{MSTTMHS}=-\eta_{k}g_{k+1}+\dfrac{g_{k+1}^Ty_{k}}{s_{k}^Ty_{k}+\zeta\Vert y_{k}\Vert\Vert s_{k}\Vert}s_{k}-\dfrac{g_{k+1}^Ts_{k}}{s_{k}^Ty_{k}+\zeta\Vert y_{k}\Vert\Vert s_{k}\Vert}y_{k},\qquad \forall k\geq0,
\end{equation}
starting with $d_{0}=-g_{0}$, where $\eta_{k}>0$ is the spectral parameter. Next, we describe a judicious scheme for adaptively setting $\eta_{k}$, benefiting from the QN aspects.

QN methods represent one of the most widely used categories of iterative algorithms for solving \eqref{UP}, characterized by approximately employing the second derivatives of the cost function, superlinear convergence, and desirable computational performance. The direction used for iterative searching in QN algorithms is formulated by
\begin{equation}\label{dQN}
d_{0}=-g_{0},\qquad d_{k+1}^\textrm{QN}=-\mathrm{H}_{k+1}g_{k+1},\qquad \forall k\geq0,
\end{equation}
in which $\mathrm{H}_{k+1}$ corresponds to an approximate form of the matrix $\nabla^2f(x_{k+1})^{-1}$. The QN methods benefit from the second derivatives basically through the secant equation \cite{Andrei}, i.e. $\mathrm{H}_{k+1}y_k = s_k$, by a recursive updating formula. Among the QN updating formulas, the BFGS and DFP formulas are the most well-known \cite{Andrei}.

As is known, the accuracy of the (inverse) Hessian approximations of the QN algorithms is closely connected to the quality of the corresponding secant equation. In recent decades, several adaptations of the traditional secant equation have been suggested \cite{Babaie}, with the aim of taking advantage of the cost function values, establishing convergence without any assumption of convexity, and employing the data available from more than one recent iteration. Rather than studies that have been founded upon the Taylor expansion, here we take care of the formal similarity between the HS formula and the memoryless BFGS updating formula, as well as the structure of $\beta_k^\textrm{MHS}$ given by (\ref{dMHS}), to suggest the extended secant equation $\mathrm{H}_{k+1}z_k = s_k$, with
\begin{equation*}\label{zeta}
   z_k = y_k + \zeta \dfrac{\left\Vert y_k\right\Vert}{\left\Vert s_k\right\Vert}s_k.
\end{equation*}

The suggested equation can also be regarded as a member of the class of modified secant equations proposed by Li and Fukushima \cite{LiFukushima}. Note that, $s_k^Tz_k$ is exactly the denominate of $\beta_k^\textrm{MHS}$ in the sense that $\beta_k^\textrm{MHS}={\left(g_{k+1}^Ty_{k}\right)}/{\left(s_k^Tz_k\right)}$. Hence, again referring to the structural similarity of the HS and BFGS search directions, we give the following modified version of the memoryless BFGS (MMLBFGS) formula, as an extension of the MHS search direction matrix:
\begin{equation*}
\mathrm{H}_{k+1}^\textrm{MMLBFGS}=\mathrm{I}-\dfrac{s_{k}y_{k}^T+y_{k}s_{k}^T}{s_{k}^Tz_{k}}+\left(1+\dfrac{\Vert y_{k}\Vert^2}{s_{k}^Tz_{k}}\right)\dfrac{s_{k}s_{k}^T}{s_{k}^Tz_{k}}.
\end{equation*}
Moreover, the MSTTMHS search direction given by (\ref{dSTTMHS}) is now immediately rewritten as
\begin{equation}\label{dnew}
d_{k+1}^\textrm{MSTTMHS}=-\eta_{k}g_{k+1}+\dfrac{g_{k+1}^Ty_{k}}{s_{k}^Tz_{k}}s_{k}-\dfrac{g_{k+1}^Ts_{k}}{s_{k}^Tz_{k}}y_{k},\qquad \forall k\geq0.
\end{equation}
The remaining issue is now attaining a meaningful formula for computing the spectral parameter $\eta_{k}$.

Advantages of the BFGS method in various computational and theoretical aspects, compared to the other QN algorithms, are reasonably enough to motivate us to determine the spectral parameter $\eta_{k}$ in \eqref{dnew} by adjusting the MSTTMHS direction to be nearly similar to the memoryless BFGS direction. To this end, in accordance with the norms of the literature, a straightforward and helpful (solution-oriented) plan is to benefit from the least-squares models \cite{DaiKou}, an approach that opens up opportunities to boost the empirical performance of the MSTTMHS and, consequently, supports the flexibility of the algorithm. So, here, in a vector model, the parameter $\eta_{k}$ is computed by solving
\begin{equation}\label{mind}
\min_{\eta_{k}}\psi(\eta_{k}):=\left\Vert d_{k+1}^\textrm{MSTTMHS}-d_{k+1}^\textrm{MMLBFGS}\right\Vert^{2},
\end{equation}
where, considering \eqref{dQN}, is equivalent to
\begin{equation*}
\min_{\eta_{k}}\psi(\eta_{k})=\left\Vert g_{k+1}\right\Vert^2\eta_{k}^2-2\eta_{k}\left(g_{k+1}^T\mathrm{H}_{k+1}^\textrm{MMLBFGS}g_{k+1}\right)+\xi,
\end{equation*}
in which $\xi$ is a real-valued constant not dependent on $\eta_{k}$. Now, the unique solution of the minimization problem \eqref{mind} can be formulated by
\begin{equation}\label{eta1}
\eta_{k}^*=1-2\dfrac{(g_{k+1}^Ty_{k})(g_{k+1}^Ts_{k})}{\Vert g_{k+1}\Vert^2(s_{k}^Tz_{k})}+\left(1+\dfrac{\Vert y_{k}\Vert^2}{s_{k}^Tz_{k}}\right)\dfrac{(g_{k+1}^Ts_{k})^2}{\Vert g_{k+1}\Vert^2(s_{k}^Tz_{k})}.
\end{equation}

Since the algorithm that we are constructing has strong technical ties with the spectral gradient algorithm, the remaining concern is to enforce positivity of the parametric setting of (\ref{eta1}). Meanwhile, the classic results emphasize that the spectral parameter should be uniformly bounded. Thus, to overcome these concerns, we can set
\begin{equation}\label{etabound}
\bar{\eta}_{k}=\max\left\lbrace m,\min\left\lbrace M,\eta_{k}^*\right\rbrace\right\rbrace,
\end{equation}
where $m$ and $M$ are respectively some prespecified small and large positive constants.

\section{A spectral three-term conjugate gradient algorithm: Analysis of descent and convergence characteristics}\label{algorithm}

Here, based on the available results (of the previous section), we first hierarchically summarize the organization of our spectral CG algorithm. Next, we prove the sufficient descent condition (\ref{sufficient descent condition 1}) for the given algorithm and present the analysis resulting in global convergence.

\begin{algorithm}[ht!]
\caption{\textbf{(MSTTMHS)}: A Modified Spectral Three-Term Extension of MHS }
\label{Algorithm 1}
\begin{algorithmic}[1] 
\Require Initial guess $x_0 \in \mathbb{R}^n$,
  MHS constant $\zeta > 0$, line search constants $\delta \in (0,1)$ and $\sigma \in (\delta,1)$, lower bound $m > 0$ and upper bound $M > 0$ for the spectral parameter, and tolerance $\varepsilon > 0$.
\State $d_0 \gets -g_0$
\State $k \gets 0$
\While{$\|g_k\| \geq \varepsilon$}
    \State Calculate $\alpha_k$ (by a line search) satisfying the Wolfe conditions, i.e.
    \begin{eqnarray}
   \label{wolfe1}    &\ & f(x_k + \alpha_k d_k) - f(x_k) \le \delta\,\alpha_k\left(g_k^T d_k\right), \\
  \label{wolfe2}     &\ &  \nabla f(x_k + \alpha_k d_k)^T d_k\ge \sigma \left(g_k^T d_k\right).
    \end{eqnarray}
    \State Update the next iterate using (\ref{xk}).
    \State Compute the spectral parameter using (\ref{etabound}).
    \State Calculate the search direction using (\ref{dnew}).
    \State $k \gets k + 1$
\EndWhile
\State \Return $x_k$
\end{algorithmic}
\end{algorithm}

A reasonable expectation of the current unconstrained optimization algorithms is to guarantee the sufficient descent condition (\ref{sufficient descent condition 1}). The next lemma demonstrates such a feature for MSTTMHS, needless to any convexity assumption, being one of the main goals of our study as well.

\begin{lemma}\label{sufficient}
For Algorithm \ref{Algorithm 1} (MSTTMHS), the sufficient descent condition \eqref{sufficient descent condition 1} holds.
\end{lemma}
\begin{proof}
For $k=0$, inequality \eqref{sufficient descent condition 1} intrinsically holds in the equality form with $\varsigma = 1$, since the first direction is actually the steepest descent (negative gradient) direction. Also, when $k\geq0$, taking the inner product of \( g_{k+1} \) with both sides of \eqref{dnew}, we have
\begin{equation*}
    g_{k+1}^Td_{k+1}^\textrm{MSTTMHS}=-\bar{\eta}_{k}\Vert g_{k+1}\Vert^2+\dfrac{g_{k+1}^Ty_k}{s_k^Tz_k}g_{k+1}^Ts_{k}-\dfrac{g_{k+1}^Ts_k}{s_k^Tz_k}g_{k+1}^Ty_{k}\leq-m\Vert g_{k+1}\Vert^2.
\end{equation*}
So, the setting $\varsigma = \min\{m,1\}$ ensures \eqref{sufficient descent condition 1}.
\end{proof}

The sufficient descent condition is generally expected to lead to global convergence. To address this issue for Algorithm \ref{Algorithm 1}, the subsequent standard assumptions are required as well.

\begin{assumption} \label{Assum1}
 $\mathcal{A}_0=\left\lbrace x\in \mathbb{R}^n| f (x) \leq f (x_{0})\right\rbrace$ is a bounded set.
\end{assumption}

\begin{assumption} \label{Assum2}
There is some neighborhood $\mathcal{S}$ containing $\mathcal{A}_0$ in which $\nabla f$ is Lipschitz continuous; i.e.,
\begin{equation*}\label{lipschitz}
\Vert \nabla f(x)-\nabla f(\breve{x})\Vert\leq L\Vert x-\breve{x}\Vert,\qquad \forall x,\breve{x}\in \mathcal{S},
\end{equation*}
for a positive (Lipschitz) constant $L$.
\end{assumption}

Note that the condition (\ref{wolfe1}) keeps the iterations in the level set $\mathcal{A}_0$, and the condition (\ref{wolfe2}) ensures that $s_{k}^Ty_{k}>0$. Also, under Assumptions \ref{Assum1} and \ref{Assum2}, it can be observed that there exists a constant $\tau>0$ which makes a bound for the norm of the gradient within $\mathcal{A}_0$ \cite{BabaieFatemiMahdavi}; i.e.
\begin{equation}\label{lip}
\Vert \nabla f(x)\Vert\leq\tau,\qquad \forall x\in \mathcal{A}_0.
\end{equation}
The following result on global convergence can now be established, without any convexity assumption.

\begin{theorem}
If Assumptions \ref{Assum1} and \ref{Assum2} hold, then Algorithm \ref{Algorithm 1} converges in the sense that $\displaystyle\liminf_{k\rightarrow\infty}\Vert g_{k}\Vert = 0$.
\end{theorem}

\begin{proof}
It follows from the Cauchy--Schwarz inequality that
\begin{eqnarray*}
   \left\Vert d_{k+1}^\textrm{MSTTMHS}\right\Vert&=&\left\Vert-\bar{\eta}_{k}g_{k+1}
   +\dfrac{g_{k+1}^Ty_{k}}{s_{k}^Tz_{k}}s_{k}-\dfrac{g_{k+1}^Ts_{k}}{s_{k}^Tz_{k}}y_{k}\right\Vert\\  \\
   &\leq& \bar{\eta}_{k}\left\Vert g_{k+1}\right\Vert+\dfrac{\left\Vert g_{k+1}\right\Vert\left\Vert y_{k}\right\Vert}{s_{k}^Ty_{k}+\zeta\Vert y_k\Vert\Vert s_k\Vert}\left\Vert s_{k}\right\Vert\\
   &&\hspace{1.7cm} +\dfrac{\left\Vert g_{k+1}\right\Vert\left\Vert s_{k}\right\Vert}{s_{k}^Ty_{k}+\zeta\Vert y_k\Vert\Vert s_k\Vert}\left\Vert y_{k}\right\Vert.
\end{eqnarray*}
Now, since $s_{k}^Ty_{k}>0$, from \eqref{etabound} and \eqref{lip}, we have
\begin{equation*}
    \left\| d_{k+1}^\textrm{MSTTMHS}\right\|\leq\left(M+\dfrac{2}{\zeta}\right)\tau,
\end{equation*}
which ensures that the search directions are uniformly bounded from above. So, by Lemma 3.1 of \cite{SugikiNarushimaYabe}, the desired result is thus proven.
\end{proof}

\section{Experimental results}\label{numerical}

In this section, we investigate the computational efficiency of Algorithm \ref{Algorithm 1}. To this aim, we compare MSTTMHS, with three other iterative methods of the form (\ref{xk}) characterized by MTTMHS with the direction (\ref{dnew}) in which $\eta_k=1$, the ZZL method with the search direction (\ref{dZZLHS}), and the CG method MHS with the direction (\ref{dk}) wherein the CG parameter is calculated by (\ref{dMHS}). In our implementations, we let $M = {m}^{-1} = 10^{-8}$ in \eqref{etabound}, and $\zeta=1.1$.

\begin{figure}[!ht]
\centering
\subfloat[TNFGE]
{\includegraphics*[width=.8\textwidth]
{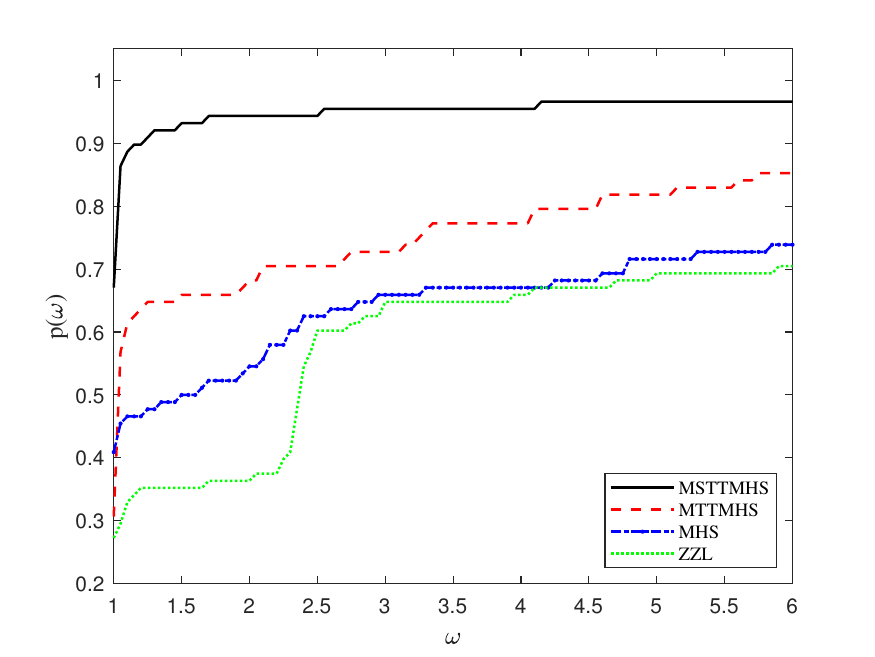}}
\hspace{3mm}
\subfloat[Running time]
{\includegraphics*[width=.8\textwidth]
{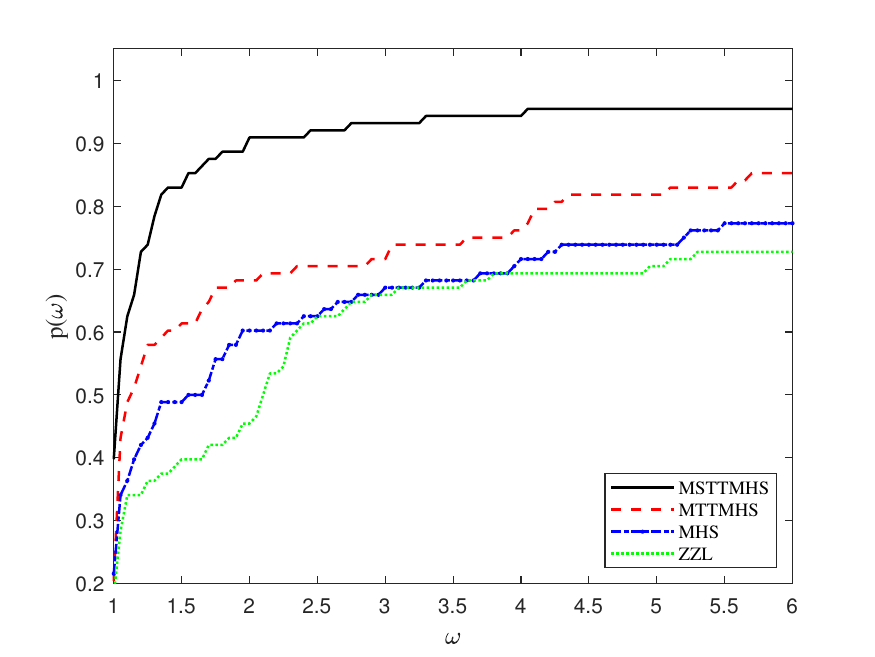}}
\caption{Results of computational comparisons on the CUTEr test problems}
\label{mylabel1}
\end{figure}

To provide supporting evidence for the numerical improvements of MSTTMHS, here, we execute numerical tests on a class of CUTEr test problems \cite{CUTEr} whose specifications have been given in \cite{Khoshsima}. In this regard, all the codes have been implemented by a system whose details have been declared in \cite{Khoshsima}. The Wolfe line search conditions \eqref{wolfe1} and \eqref{wolfe2} have been incorporated into our implementations with $(\delta,\sigma)=(0.0001,0.9900)$. The iterative process terminated when the iteration number $k$ exceeds $10000$, or $\Vert g_k\Vert_\infty<10^{-6}(1+\vert f(x_k)\vert)$. To assess the quality of the outputs, we used the standard performance profile of \cite{DolanMore}, on TNFGE (the total number of function and gradient evaluations \cite{HagerZhang}), and the running time. The outputs have been visualized by Figure \ref{mylabel1}.

As seen, the MSTTMHS method demonstrates superior performance compared to the other algorithms. In addition, the outputs show that all the modified methods (with an augmented denominator) outperform the classic three-term ZZL method. This fact emphasizes that, generally, our modification scheme, which originates from the modification approach of Chen, Kuang, and Yan \cite{ChenKuangYan} conducted on the classic PRP method, turns out to be effective in practice. The plots also depict the promising numerical performance of the modified three-term method MTTMHS versus the two-term modified HS method MHS.

\section{A reduced quadratic surface SVM with the rescaled loss for robust binary classification}
\label{sec:rqsvm}

In this section, we introduce a reduced quadratic surface SVM model incorporating the rescaled loss for robust binary classification. Using Algorithm~1 to solve the resulting smooth training problem as an unconstrained optimization model, we compare the proposed model with related SVMs in terms of classification accuracy and training time.

Consider a binary classification problem with training data 
$\mathcal{D}=\{(x_i,y_i)\}_{i=1}^{m}$, where $x_i\in\mathbb{R}^n$ and 
$y_i\in\{-1,+1\}$. Classical support vector machines construct linear 
separators or employ kernel mappings to capture nonlinear patterns. Although 
kernel methods offer considerable flexibility, their performance depends 
strongly on the choice of kernel, and they may lead to models that are difficult 
to interpret and computationally expensive for large datasets.
An alternative approach is to construct nonlinear decision boundaries directly 
in the input space. Generally, a quadratic surface support vector machine (QSVM) follows
this strategy by employing the quadratic decision function
\[
f_{W,b,c}(x)=\frac{1}{2}x^{T}Wx+b^{T}x+c,
\]
where $W\in\mathbb{R}^{n\times n}$ is symmetric, $b\in\mathbb{R}^n$, and 
$c\in\mathbb{R}$. This kernel-free formulation retains the ability to model 
nonlinear decision boundaries without relying on implicit feature mappings.

A fully dense matrix $W$ introduces $O(n^2)$ parameters, which can lead to
overfitting and poor interpretability. Various sparsifying strategies,
including $\ell_1$-regularization \cite{moosaei2023sparse,mousavi2022quadratic},
structured sparsity, and explicit $\ell_0$-constraints \cite{mousavi2025l0},
aim to mitigate this issue. Here, we adopt a reduced formulation by
imposing a diagonal structure on $W$, i.e., $W=\mathrm{Diag}(w)$ for
$w\in\mathbb{R}^n$ \cite{GaoWang}. This eliminates interaction terms while
retaining nonlinear feature-wise curvature. With this structure, the
decision function becomes
\begin{equation*}
f(x)=\frac{1}{2}\sum_{j=1}^{n}w_jx_j^2+b^{T}x+c.
\end{equation*}

Define the least-squares residual $\gamma_i=1-y_i f(x_i)$, $i=1,2,...,m$. Squared losses
penalize large residuals aggressively; we therefore use the rescaled loss
of \cite{Aminifard2027}, i.e.,
\[
1-\exp(-\eta\gamma_i^2),\qquad i=1,2,...,m,
\]
with the constant $\eta>0$, which is quadratic near zero and saturates for large $|\gamma_i|$. Because
the loss depends on $\gamma_i^2$, it represents a bounded least-squares
residual rather than a one-sided hinge loss. The resulting model, denoted
by R--RQSVM in analogy with the rescaled (R) twin SVM models, is
\begin{equation*}
\min_{w,b,c}
\sum_{i=1}^{m}
\bigl\|\mathrm{Diag}(w)x_i+b\bigr\|_2^2
+C\sum_{i=1}^{m}
\bigl(1-\exp(-\eta\gamma_i^2)\bigr),
\end{equation*}
where $C>0$ is the penalty parameter.

To obtain a vectorized quadratic representation, define
\[
q_i:=\frac{1}{2}(x_i\odot x_i),\qquad
r_i:=\begin{bmatrix}q_i\\ x_i\end{bmatrix},\qquad
z:=\begin{bmatrix}w\\ b\end{bmatrix}\in\mathbb{R}^{2n},
\]
so that $f(x_i)=z^{T}r_i+c$. Moreover,
\[
Wx_i+b=\mathrm{Diag}(x_i)w+b=H_i z,\qquad
H_i:=\bigl[\mathrm{Diag}(x_i)\ \ I_n\bigr]
\in\mathbb{R}^{n\times 2n}.
\]
Thus,
\begin{equation*}
\sum_{i=1}^{m}\|Wx_i+b\|_2^2
=\sum_{i=1}^{m}\|H_i z\|_2^2
=\frac{1}{2}z^{T}Gz,\qquad
G:=2\sum_{i=1}^{m}H_i^{T}H_i\succeq 0.
\end{equation*}
Since $z\in\mathbb{R}^{2n}$ and $H_i\in\mathbb{R}^{n\times 2n}$, one has
$G\in\mathbb{R}^{2n\times 2n}$. Compared with the full QSVM, the reduced
model scales linearly with $n$. The compact formulation, here called R--RQSVM$'$, is defined by
\begin{equation*}
\min_{z,c}
\frac{1}{2}z^{T}Gz
+C\sum_{i=1}^{m}
\left(
1-\exp\left[
-\eta\bigl(1-y_i(z^{T}r_i+c)\bigr)^2
\right]
\right).
\end{equation*}
The objective function in R--RQSVM$'$ is unconstrained and
continuously differentiable. It can therefore be minimized by Algorithm~1,
and the sufficient descent and global convergence properties established
in Section~3 apply without requiring a convexity assumption. Furthermore,
a single kernel-free quadratic surface with $O(n)$ parameters yields a
smooth training problem of the type handled by Algorithm~1, in contrast
to the twin linear SVMs.

Next, we compare R--RQSVM with four linear least-squares twin SVM models
that employ either the rescaled (R) or rational (Ra) loss: R--LS--TSVM,
Ra--LS--TSVM, R--LS--GTPSVM, and Ra--LS--GTPSVM \cite{Aminifard2027}.
Table~\ref{tab:r-rqsvm-performance} reports the classification accuracy and
CPU time on eight standard datasets, while Table~\ref{tab:friedman_time}
reports the Friedman ranks for CPU time, in which rank $1$ indicates the best
performance.

As the outputs show, in terms of accuracy, R--RQSVM is competitive with the four linear twin SVM models, even though each twin baseline fits two nonparallel linear classifiers, whereas R--RQSVM trains a single reduced quadratic surface. Across the eight datasets, the accuracies of all five methods lie within a common range, and no
method uniformly dominates the others; we therefore do not claim a predictive
advantage for R--RQSVM on this suite. In terms of computational cost, the
distinction is clear: R--RQSVM is the fastest method on all eight datasets and
obtains the lowest mean rank (i.e., $1.00$). A Friedman test on CPU time yields
$\chi_F^2=24.80$ with $df=4$ and $p=5.52\times10^{-5}$, so the null hypothesis
of equal running times is rejected. Part of this gap is structural: R--RQSVM
solves a single training problem, whereas the twin models solve either two
class-wise problems or one coupled two-plane problem. Taken together, these
results show that a single reduced quadratic surface trained with the rescaled
loss attains accuracy comparable to that of linear twin SVMs while requiring
less training time.

\begin{table}[H]
\centering
\caption{Classification accuracy and CPU time of R--RQSVM and four
least-squares twin SVM models}
\label{tab:r-rqsvm-performance}
\begin{tabular}{llcc}
\specialrule{1.5pt}{0pt}{0pt}
\textbf{Dataset} & \textbf{Model} & \textbf{Accuracy (\%)} & \textbf{Time (s)}\\
\specialrule{1.5pt}{0pt}{0pt}
\multirow{5}{*}{Iris}
& R--LS--TSVM    & $100.00 \pm 0.00$ & $0.58$ \\
& Ra--LS--TSVM   & $100.00 \pm 0.00$ & $0.61$ \\
& R--LS--GTPSVM  & $100.00 \pm 0.00$ & $0.57$ \\
& Ra--LS--GTPSVM & $100.00 \pm 0.00$ & $0.65$ \\
& R--RQSVM        & $100.00 \pm 0.00$ & $0.18$ \\
\specialrule{.5pt}{0pt}{0pt}
\multirow{5}{*}{Hayes--Roth}
& R--LS--TSVM    & $74.57 \pm 0.25$ & $0.85$ \\
& Ra--LS--TSVM   & $75.76 \pm 0.28$ & $0.82$ \\
& R--LS--GTPSVM  & $73.00 \pm 0.32$ & $0.91$ \\
& Ra--LS--GTPSVM & $74.00 \pm 0.30$ & $0.89$ \\
& R--RQSVM        & $74.60 \pm 0.77$ & $0.73$ \\
\specialrule{.5pt}{0pt}{0pt}
\multirow{5}{*}{Diabet}
& R--LS--TSVM    & $75.78 \pm 0.35$ & $3.35$ \\
& Ra--LS--TSVM   & $77.09 \pm 0.32$ & $2.92$ \\
& R--LS--GTPSVM  & $75.10 \pm 0.38$ & $3.21$ \\
& Ra--LS--GTPSVM & $75.23 \pm 0.36$ & $3.49$ \\
& R--RQSVM        & $76.31 \pm 0.33$ & $1.53$ \\
\specialrule{.5pt}{0pt}{0pt}
\multirow{5}{*}{Cancer}
& R--LS--TSVM    & $94.13 \pm 0.22$ & $2.25$ \\
& Ra--LS--TSVM   & $92.38 \pm 0.25$ & $2.10$ \\
& R--LS--GTPSVM  & $95.52 \pm 0.20$ & $2.40$ \\
& Ra--LS--GTPSVM & $94.52 \pm 0.23$ & $2.30$ \\
& R--RQSVM        & $95.28 \pm 0.21$ & $1.17$ \\
\specialrule{.5pt}{0pt}{0pt}
\multirow{5}{*}{Heart}
& R--LS--TSVM    & $84.89 \pm 0.42$ & $5.25$ \\
& Ra--LS--TSVM   & $82.96 \pm 0.45$ & $4.40$ \\
& R--LS--GTPSVM  & $85.85 \pm 0.38$ & $5.75$ \\
& Ra--LS--GTPSVM & $81.59 \pm 0.44$ & $5.86$ \\
& R--RQSVM        & $83.13 \pm 0.56$ & $1.70$ \\
\specialrule{.5pt}{0pt}{0pt}
\multirow{5}{*}{Australian}
& R--LS--TSVM    & $85.52 \pm 0.28$ & $8.05$ \\
& Ra--LS--TSVM   & $83.29 \pm 0.31$ & $7.78$ \\
& R--LS--GTPSVM  & $86.03 \pm 0.26$ & $9.99$ \\
& Ra--LS--GTPSVM & $84.71 \pm 0.29$ & $8.87$ \\
& R--RQSVM        & $83.13 \pm 0.56$ & $3.28$ \\
\specialrule{.5pt}{0pt}{0pt}
\multirow{5}{*}{Sonar}
& R--LS--TSVM    & $85.76 \pm 0.52$ & $12.75$ \\
& Ra--LS--TSVM   & $82.66 \pm 0.55$ & $10.87$ \\
& R--LS--GTPSVM  & $81.31 \pm 0.58$ & $14.12$ \\
& Ra--LS--GTPSVM & $82.53 \pm 0.54$ & $11.61$ \\
& R--RQSVM        & $83.38 \pm 0.30$ & $4.28$ \\
\specialrule{.5pt}{0pt}{0pt}
\multirow{5}{*}{Splice}
& R--LS--TSVM    & $81.00 \pm 0.48$ & $20.93$ \\
& Ra--LS--TSVM   & $81.60 \pm 0.46$ & $20.25$ \\
& R--LS--GTPSVM  & $80.30 \pm 0.51$ & $23.26$ \\
& Ra--LS--GTPSVM & $80.10 \pm 0.49$ & $22.42$ \\
& R--RQSVM        & $82.50 \pm 0.40$ & $4.08$ \\
\specialrule{1.5pt}{0pt}{0pt}
\end{tabular}
\end{table}

\begin{table}[H]
\centering
\scriptsize
\setlength{\tabcolsep}{3pt}
\caption{Friedman ranks based on CPU execution time}
\label{tab:friedman_time}
\begin{tabular}{@{}lccccc@{}}
\specialrule{1.5pt}{0pt}{0pt}
\textbf{Dataset} &
\textbf{R--LS--TSVM} &
\textbf{Ra--LS--TSVM} &
\textbf{R--LS--GTPSVM} &
\textbf{Ra--LS--GTPSVM} &
\textbf{R--RQSVM} \\
\specialrule{1pt}{0pt}{0pt}
Iris        & $3$ & $4$ & $2$ & $5$ & $1$ \\
Hayes--Roth & $3$ & $2$ & $5$ & $4$ & $1$ \\
Diabet      & $4$ & $2$ & $3$ & $5$ & $1$ \\
Cancer      & $3$ & $2$ & $5$ & $4$ & $1$ \\
Heart       & $3$ & $2$ & $4$ & $5$ & $1$ \\
Australian  & $3$ & $2$ & $5$ & $4$ & $1$ \\
Sonar       & $4$ & $2$ & $5$ & $3$ & $1$ \\
Splice      & $3$ & $2$ & $5$ & $4$ & $1$ \\
\specialrule{1pt}{0pt}{0pt}
\textbf{Mean rank} &
$\mathbf{3.25}$ &
$\mathbf{2.25}$ &
$\mathbf{4.25}$ &
$\mathbf{4.25}$ &
$\mathbf{1.00}$ \\
\specialrule{.5pt}{0pt}{0pt}
\multicolumn{6}{c}{
$\chi_F^2=24.80$,\quad $df=4$,\quad $p=5.52\times 10^{-5}$
} \\
\specialrule{1.5pt}{0pt}{0pt}
\end{tabular}

\end{table}

\section{Conclusions and future works}\label{conclusions}
We have suggested a class of three-term conjugate gradient algorithms that benefit from a spectral parameter, to be set by employing the approximate second-order information. Our modification plan originates from a recent modified Polak--Ribi\`{e}re--Polyak parameter proposed by Chen, Kuang, and Yan \cite{ChenKuangYan}. Actually, suggesting a revised version of the Hestenes--Stiefel parameter, we have developed a modified secant equation that effectively helps to determine the spectral parameter in the framework of a least-squares model. Our algorithm preserves the sufficient descent property and, as a result, it converges globally for general cost functions. At the same time, the algorithm is capable of performing approximate restarts when the iterations accumulate far from the solution. In another part of our study, we have introduced a revised reduced quadratic surface SVM (R--RQSVM), for robust binary classification. The proposed model combines a diagonal quadratic decision surface with the rescaled loss, yielding a kernel-free, smooth, and unconstrained training problem with only $O(n)$ parameters, which can be efficiently solved by the suggested three-term conjugate gradient algorithm. In contrast to linear twin SVMs that require two nonparallel classifiers, R--RQSVM employs a single reduced quadratic surface while retaining nonlinear classification capability.

Although both the theoretical analysis and numerical experiments yield satisfactory results, it can be a matter of further research to study other possible settings for the spectral parameter. In this regard, we have utilized the ellipsoid vector norm within our least-squares model, as an extension of the $\ell_2$ (Euclidean) norm \cite{MMA}. By the way, the obtained formulas were to some extent complex, an issue that may heavily influence the computational cost of the algorithm, especially in high dimensions. Utilizing the cost function values alongside the first-order derivative in the algorithm could enhance the accuracy as well, an objective that may be available by modifying the secant equation. Regarding to the given revised SVM model, future research may focus on extending the proposed R--RQSVM model by considering more flexible structured forms of the quadratic matrix of the model, while preserving its reduced computational complexity, and further investigating its robustness and scalability on high-dimensional and large-scale classification problems.

\section*{Declarations}
\noindent
\textbf{Ethics approval and consent to participate.} This study did not involve human participants or animals and therefore did not require ethical approval.\\

\noindent
\textbf{Consent for publication.} This study does not contain any identifiable data pertaining to individuals or organizations and therefore does not require consent for publication.\\

\noindent
\textbf{Availability of data and material.} The datasets and materials used and/or analyzed during the current study can be available from the corresponding author on reasonable request.\\

\noindent
\textbf{Competing interests.} There are no competing interests to declare by the authors with respect to this article.\\

\noindent
\textbf{Funding.}  No funding, grants, or other support were received.\\

\noindent
\textbf{Authors' contributions.} The third author was primarily responsible for Section 5, while the first two authors jointly contributed to most of the other parts of the study.\\

\noindent
\textbf{Acknowledgements.} The authors thank Dr. Reza Ghanbari from Ferdowsi University of Mashhad (IRAN) for his consultations in preparing the codes.


\end{document}